\documentclass[10pt]{amsart}

\usepackage[english,french]{babel}
\usepackage[T1]{fontenc}
\usepackage{amsmath,amssymb}
\usepackage[all]{xy}
\usepackage{graphicx}
\usepackage{fancyhdr}

\newtheorem{definition}{D\'efinition }[section]

\newtheorem{lemma}[definition]
{Lemme}
\newtheorem{theorem}[definition]
{Th\'eor\`eme }
\newtheorem{ex}[definition]
{Exemple }

\newtheorem{prop}[definition]{Proposition}

\newcommand{\lgw}{\longrightarrow}
\newcommand{\lgm}{\longmapsto}

\newcommand{\ovl}{\overline}

\renewcommand{\deg}{\text{deg}\,}
\newcommand{\ord}{\text{ord}}
\newcommand{\Ker}{\text{Ker}}

\newcommand{\wdt}{\widetilde}
\renewcommand{\l}{\lambda}

\renewcommand{\k}{\Bbbk}

\newcommand{\N}{\mathbb{N}}

\renewcommand{\a}{\alpha}
\renewcommand{\b}{\beta}

\renewcommand{\phi}{\varphi}
\renewcommand{\d}{\delta}

\begin{document}
\title{Bornes effectives des fonctions d'approximation des solutions formelles d'\'equations binomiales}
\author{Guillaume Rond}

\address{IML, Campus de Luminy, Case 907, 
13288 Marseille Cedex 9, France}
\email{rond@iml.univ-mrs.fr}
\begin{abstract}
The aim of this paper is to give an effective version of the Strong Artin Approximation Theorem for binomial equations. First we give an effective version of the Greenberg Approximation Theorem for polynomial equations, then using the Weierstrass Preparation Theorem, we apply this effective result to binomial equations. We prove that the Artin function of a system of binomial equations is bounded by a doubly exponential function in general and that it is bounded by an affine function if the the order of the approximated solutions is bounded.
\end{abstract}

\maketitle
\section{Introduction}
Dans \cite{Gr}, M. Greenberg a montr\'e le th\'eor\`eme suivant ($\k$ est un corps quelconque) :
\begin{theorem}\label{Gr}\cite{Gr}
Soit $I$ un id\'eal de $\k[t,X]$ ($X:=(X_1,...,X_n)$). Alors il existe une fonction  $\b :\N\lgw \N$  telle que :
$$\forall i\in\N, \forall x_j(t)\in\k[[t]], 1\leq j\leq n \text{ tels que } f(t,x(t))\in (t)^{\b(i)} \ \forall f\in I$$
$$\exists \ovl{x}_j(t)\in\k[[t]], 1\leq j\leq n,\text{ tels que } f(t,\ovl{x}(t))=0 \ \forall f\in I\text{ et } \ovl{x}_j(t)-x_j(t)\in (t)^i \ \forall j.$$
De plus il existe deux constantes $a$ et $b$ telles que $\forall i\in\N$ $\b(i)\leq ai+b$. 
\end{theorem}
La plus petite fonction $\b$ v\'erifiant cette propri\'et\'e est appel\'ee fonction d'approximation de Greenberg ou fonction de Artin-Greenberg de $I$.
Dans \cite{Ar2}, M. Artin a g\'en\'eralis\'e ce r\'esultat au cas o\`u $t$ est remplac\'e par un nombre fini quelconque de variables en montrant l'existence d'une fonction d'approximation dans ce cas. On peut \'enoncer son th\'eor\`eme sous la forme suivante (dans le cas particulier o\`u $t$ est remplac\'e par deux variables $t$ et $z$, mais seul ce cas nous int\'eressera par la suite) :
\begin{theorem}\label{Ar}\cite{Ar2}
Soit $I$ un id\'eal de $\k[t,z,X]$ ($X:=(X_1,...,X_n)$). Alors il existe une fonction $\b : \N\lgw \N$ telle que :
$$\forall i\in\N, \forall x_j(t,z)\in\k[[t,z]], 1\leq j\leq n, \text{ tels que } f(t,z,x(t,z))\in (t,z)^{\b(i)} \ \forall f\in I$$
$$\exists \ovl{x}_j(t,z)\in\k[[t,z]], 1\leq j\leq n\text{ tels que } f(t,z,\ovl{x}(t,z))=0 \ \forall f\in I\text{ et } \ovl{x}_j(t,z)-x_j(t,z)\in (t,z)^i \ \forall j.$$
\end{theorem}
La plus petite fonction v\'erifiant cette propri\'et\'e est appel\'ee fonction de Artin de $I$. On peut remarquer que si $I$ est un id\'eal de $\k[X]$ on peut parler de sa fonction de Artin-Greenberg et de sa fonction de Artin en faisant r\'ef\'erence \`a la fonction de Artin-Greenberg de $I\k[t,X]$ et \`a la fonction de Artin de $I\k[t,z,X]$.\\
 Malheureusement, la preuve donn\'ee par M. Artin (qui utilise essentiellement le th\'eor\`eme des fonctions implicites et le th\'eor\`eme de pr\'eparation de Weierstrass) n'apporte que tr\`es peu d'information sur la nature de la croissance de cette fonction (except\'ee que celle-ci est constructive \cite{La}), et pendant longtemps s'est pos\'ee la question de savoir si toute fonction de Artin \'etait born\'ee par une fonction affine. On peut mentionner qu'il existe une mani\`ere diff\'erente de celle de M. Artin pour montrer l'existence d'une fonction d'approximation qui utilise les ultraproduits \cite{ultra I} et qui n'est pas constructive et donc qui n'apporte aucune information sur ces fonctions d'approximation.  Dans \cite{Ro2} est pr\'esent\'e un exemple d'un id\'eal dont la fonction de Artin n'est pas born\'ee par une fonction affine (c'est le seul exemple connu jusqu'\`a pr\'esent, l'exemple donn\'e dans \cite{Ro3} concernant le cas o\`u les variables $t$ et $z$ sont remplac\'ees par trois variables). Dans cet exemple il est montr\'e que la fonction de Artin consid\'er\'ee est minor\'ee par une fonction polynomiale de degr\'e 2. N\'eanmoins on ne sait rien de plus sur cette fonction de Artin et on ne connait toujours aucune borne g\'en\'erale sur aucun exemple autre que quelques cas connus o\`u la fonction de Artin est major\'ee par une fonction affine. Pour mieux comprendre la croissance des fonctions de Artin,  il serait int\'eressant d'avoir des exemples pour lesquels on ait des bornes effectives de leur fonction de Artin. On peut remarquer que l'exemple cit\'e de \cite{Ro2} a la particularit\'e d'\^etre un id\'eal principal engendr\'e par un bin\^ome.\\
 Le but de ce travail est justement de donner des bornes g\'en\'erales pour la fonction de Artin d'un id\'eal binomial $I$. Le principe g\'en\'eral est le suivant : il est plus facile d'essayer de donner des bornes sur les fonctions de Artin d'une famille d'id\'eaux que sur un id\'eal en particulier. En effet, la preuve de M. Artin et ses avatars utilisent toujours une r\'ecurrence sur la hauteur de $I$ : si $I$ est premier, soit on peut utiliser le th\'eor\`eme des fonctions implicites et on construit la solution approch\'ee cherch\'ee, soit on remplace $I$ par $I+(\delta)$ o\`u $\delta$ est un  mineur de la matrice jacobienne de $I$ bien choisi et on augmente ainsi la hauteur de $I$. Cependant, le nouvel id\'eal obtenu n'est plus premier a priori et    l'utilisation du th\'eor\`eme des fonctions implicites n\'ecessitant de travailler avec un id\'eal premier, il faut  remplacer $I$ par ses id\'eaux premiers associ\'es et de borner la fonction de Artin de $I$ par celles de ses id\'eaux premiers associ\'es. Il se trouve que l'on connait des bornes sur le degr\'e de g\'en\'erateurs de ces id\'eaux premiers en fonction du degr\'e des g\'en\'erateurs de $I$. Il est donc assez naturel d'essayer de trouver une fonction qui majore toutes les fonctions de Artin des id\'eaux engendr\'e par des polyn\^omes de degr\'e $d$ fix\'e. Cette strat\'egie fonctionne comme l'a montr\'e M. Artin, i.e. on peut choisir la m\^eme fonction $\b$ dans le th\'eor\`eme \ref{Ar} pour tous les id\'eaux engendr\'es par des polyn\^omes de degr\'e inf\'erieur \`a une valeur fix\'ee (c.f. \cite{Ar2}). N\'eanmoins, il n'est pas possible d'avoir plus d'informations sur cette fonction de Artin "uniforme" que le fait que celle-ci soit constructible (cf. \cite{La}).\\
 Dans notre travail, nous montrons n\'eanmoins que cette strat\'egie donne une borne effective des fonctions de Artin-Greenberg, c'est-\`a-dire donne une borne effective
  sur les coefficients $a$ et $b$ du th\'eor\`eme \ref{Gr} en fonction du degr\'e $d$ des g\'en\'erateurs de $I$ et du nombre $n$ de variables $X_i$ (c.f. th\'eor\`eme \ref{theo1}).  Celle-ci est assez grande (polynomiale en $d$, doublement exponentielle en $n$) mais a le m\'erite d'\^etre uniforme. Ensuite, gr\^ace au lemme \ref{w}, nous utilisons le th\'eor\`eme de pr\'eparation de Weierstrass de mani\`ere \`a ramener la majoration de la fonction de Artin d'un id\'eal binomial \`a deux majorations (cf. section \ref{binomsec}) :\\
 $\bullet$ La majoration de la fonction de Artin d'un id\'eal binomial dont les solutions approch\'ees \'evitent le lieu singulier (et donc dans ce cas cette fonction est major\'ee par une fonction affine).\\
 $\bullet$ La majoration des fonctions de Artin-Greenberg d'une famille d'id\'eaux $(J_D)_{D\in N}$ engendr\'es par des polyn\^omes d\'ependant d'un nombre croissant de variables mais dont le degr\'e est born\'e par le degr\'e des g\'en\'erateurs de $I$. Pour cela on utilise donc le th\'eor\`eme \ref{theo1}.\\
 On en d\'eduit deux choses (sur un corps alg\'ebriquement clos de caract\'eristique nulle) : si l'on borne l'ordre des solutions approch\'ees $x_j(t,z)$, $1\leq j\leq n$, alors la fonction de Artin d'un id\'eal binomial est major\'ee par une fonction affine, et en g\'en\'eral la fonction de Artin d'un id\'eal binomial est major\'ee par une fonction doublement exponentielle (c.f. th\'eor\`eme \ref{binom}). Malheureusement on n'a aucune mani\`ere de savoir si cette derni\`ere borne est raisonnable ou pas. En effet la famille d'id\'eaux $(J_D)$ n'est pas quelconque, mais semble tout de m\^eme assez difficile \`a appr\'ehender : ces id\'eaux ne sont pas r\'eduits en g\'en\'eral et il est vite impossible de calculer leur radical ou une d\'ecomposition primaire de  ceux-ci du fait du nombre rapidement important de variables qui entrent en jeu.

 \section{Rappels sur certaines bornes effectives en alg\`ebre commutative}
Nous allons commencer par rappeler quelques r\'esultats classiques en alg\`ebre commutative effective que nous allons utiliser librement dans la suite.
\begin{theorem}\label{bornes_alg}\cite{He}\cite{Te}
Soit $I$ un id\'eal de $\k[u]$, $u:=(u_1,...,u_n)$, $I=(f_1,...,f_p)$ avec $\deg(f_r)\leq d$ pour $1\leq r\leq p$. Soit $I=Q_1\cap\cdots\cap Q_s$ une d\'ecomposition primaire minimale (i.e telle que $\sqrt{Q_l}$,..., $\sqrt{Q_s}$ soient tous distincts). On a alors $\sqrt{I}=P_1\cap\cdots\cap P_s$ avec $P_l:=\sqrt{Q_l}$ pour $1\leq l\leq s$. On a alors les r\'esultats suivants :
\begin{enumerate}
\item[a)] Soit $e:=\min\{n,p\}(n+2)(d+1)^{\min\{n,p\}+1}\leq (n+2)^2(d+1)^{n+1}$. On a $\sqrt{I}^e\subset I$.
\item[b)] On $s\leq d^{\min\{n,p\}}$.
\item[c)] Il existe une fonction $\l_1(n,d)$, polynomiale en $d$ et de degr\'e exponentiel en $n$, tel que chaque $P_l$ est engendr\'e par des polyn\^omes de degr\'e inf\'erieur ou \'egal \`a $\l_1(n,d)$.
\item[d)] Il existe une fonction $\l_2(n,d)$, polynomiale en $d$ et de degr\'e exponentiel en $n$, tel que chaque $Q_l$ est engendr\'e par des polyn\^omes de degr\'e inf\'erieur ou \'egal \`a $\l_1(n,d)$.
\end{enumerate}
\end{theorem}
\begin{prop}\cite{S}
Soient $I_1$,..., $I_q$, $q$ id\'eaux de $\k[u_1,...,u_n]$ engendr\'es par des polyn\^omes de degr\'e inf\'erieur ou \'egal \`a $d$. Alors $I:=I_1\cap\cdots\cap I_q$ est engendr\'e par des polyn\^omes de degr\'e inf\'erieur ou \'egal \`a $n((q-1)d)^{2^{n-1}}+d$
\end{prop}

\begin{proof}
\'Ecrivons $I_i=(f_{i,1},...,\,f_{i,s_i})$ o\`u $\deg(f_{i,j})\leq d$ pour tous entiers $i$ et $j$. Tout syst\`eme de g\'enerateurs de $I$, not\'e $g^{(1)},...,g^{(s)}$ correspond \`a un syst\`eme g\'en\'erateur, not\'e $\{u^{(1)}_{i,j},...,\,u_{i,j}^{(s)}\}$ pour $1\leq i\leq q$ et $1\leq j\leq s_i$, du $\k[u_1,...,u_n]$-module d\'efini par les \'equations
$$(f_{1,\,1}U_{1,\,1}+\cdots+f_{1,s_1}U_{1,s_1})-(f_{i,1}U_{i,\,1}+\cdots+f_{i,s_i}U_{i,s_i})=0,\ \  2\leq i\leq q$$
et reli\'e par $g^{(l)}=f_{i,1}u^{(l)}_{i,\,1}+\cdots+f_{i,s_i}u^{(l)}_{i,s_i}$ pour $1\leq l\leq s$. D'apr\`es la proposition 55 de \cite{S}, il existe un tel syst\`eme $\{u^{(1)}_{i,j},...,\,u_{i,j}^{(s)}\}$ tel que $\deg(u^{(l)}_{i,j})\leq n((q-1)d)^{2^{n-1}}$ pour tous $i$, $j$ et $l$. On en d\'eduit le r\'esultat.
\end{proof}

\section{Borne effective de la fonction de Artin-Greenberg dans le cas polynomial}
Dans cette partie nous allons red\'emontrer le th\'eor\`eme de Greenberg en suivant essentiellement sa preuve mais en faisant attention \`a la complexit\'e de chaque \'etape.
Dans toute cette partie $\k$ d\'esigne un corps de caract\'eristique nulle.
\begin{theorem}\label{theo1}\cite{Gr}\cite{Ar2}\label{1var}
Pour tous $n,\,d,\,i\in\N$, il existe $\b :\N\lgw \N$ tel que pour tout id\'eal $I$ de $\k[t,x]$, avec  $x=(x_1,...,x_n)$, tel que $\deg(f_r)\leq d$ pour tout $r$, pour tout $i\in\N$ et  pour tout $x(t)\in\k[[t]]^n$ tel que $f(t,x(t))\in (t)^{\b(i)}$ pour tout $f\in I$, il existe $\ovl{x}(t)\in\k[[t]]^n$ tel que $f(t,\ovl{x}(t))=0$ pour tout $f\in I$ et $x(t)-\ovl{x}(t)\in (t)^i$.\\
De plus $\b$ peut \^etre choisie affine, de la forme $i\lgm a(n,d)i+b(n,d)$ o\`u $a(n,d)$ et $b(n,d)$ sont born\'es par une fonction polynomiale en $d$ de degr\'e exponentiel en $n$.
\end{theorem}
\begin{definition}
Nous noterons dans la suite $\b(n,d,i)$ le plus petit entier $\b(i)$ qui v\'erifie les conclusions du th\'eor\`eme \ref{theo1}. Nous allons noter ${\b}_{h}(n,d,i)$ le plus petit entier qui v\'erifie les conclusions du th\'eor\`eme \ref{theo1} pour tout id\'eal $I$ quelconque de hauteur $h$ engendr\'e par des polyn\^omes de degr\'e inf\'erieur ou \'egal \`a $d$, et ${\b}^p_{h}(n,d,i)$ le plus petit entier qui v\'erifie les conclusions du th\'eor\`eme \ref{theo1} pour tout id\'eal $I$ premier de hauteur $h$ engendr\'e par des polyn\^omes de degr\'e inf\'erieur ou \'egal \`a $d$.\\ 

\end{definition}
\begin{theorem}\label{theo3}
Pour tous $n,d,d',i,j\in\N$, nous avons les relations suivantes :
\begin{equation*}\label{2^0}{\b}^p_{n+1}(n,d,i)\leq2\tag{1}\end{equation*}
\begin{equation}\label{2}{\b}_k(n,d,i)\leq (n+3)^2(d+1)^{2n+3}\max_{h>k}\left\{{\b}_h^p\left(n,\l_1(n+1,d),i\right)\right\}\tag{2}\end{equation}
\begin{equation}\label{2'}{\b}_k^p(n,d,i)\leq (e'+1){\b}_{k+1}(n,k(d-1),i)+1\tag{3}\end{equation}
$$ e':=(n+3)^2\left(1+\l_2(n+1,d)+(n+1)((d^{n+1}-2)\l_2(n+1,d))^{2^{n}}\right)^{n+2}.$$\\

\end{theorem}
Ce th\'eor\`eme implique directement le th\'eor\`eme \ref{theo1}.  Nous allons donner une preuve du th\`eor\`eme \ref{theo1} en montrant au fur et \`a mesure  les in\'egalit\'es (\ref{2^0}), (\ref{2}) et (\ref{2'}) du th\'eor\`eme \ref{theo3}.\\ 
 \\
$\bullet$ Si $\text{ht}(I)=n+1$ et $I$ est premier, alors $I$ est un id\'eal maximal. Supposons qu'il existe $x(t)\in\k[[t]]^x$ tel que $f(t,\,x(t))\in (t)^2$ pour tout $f\in I$, et notons $\phi\ :\ \k[t,x]\lgw \k$  le $\k$-homomorphisme d\'efini par $\phi(h)=h(0,\,x(0))$. Alors $I\subset \Ker(\phi)$, mais $\Ker(\phi)$ \'etant un id\'eal propre de $\k[t,x]$ et $I$ maximal, nous avons $I=\Ker(\phi)$. Donc en particulier $t\in I$, mais clairement $t\notin \Ker(\phi')$ ce qui contredit l'existence de $x(t)$. Donc on peut prendre ici $\b=2$, le th\'eor\`eme \ref{theo1} est alors valable car l'hypoth\`ese n'est jamais v\'erifi\'ee, et on a l'\'egalit\'e (\ref{2^0}).\\
\\
$\bullet$ Supposons que $\text{ht}(I)=k$, o\`u $I=(f_1,...,f_p)$, et que le th\'eor\`eme \ref{theo1} est vrai pour tout id\'eal de hauteur strictement plus grande que $k$.\\ 
Soit $\sqrt{I}=P_1\cap\cdots\cap P_s$ la d\'ecomposition primaire du radical de $I$, o\`u $\text{ht}(P_l)=h_l$. Notons $e:=(n+1)(n+3)(d+1)^{n+2}$. Alors l'entier $\b':=e.\sum_{l=1}^s{\b }^p_{h_l}(n,\l_1(n+1,d),i)$ satisfait les conditions du th\'eor\`eme \ref{theo1} pour $I$. En effet soit $x(t)\in\k[[t]]^n$ tel que $f(t,x(t))\in (t)^{\b' }$ pour tout $f\in I$.   Donc il existe un entier $l$ tel que $g(t,x(t))\in(t)^{{\b}^p_{h_l}(n,\l(n+1,d),i)}$ pour tout $g\in P_l$. En effet, dans le cas contraire, pour tout $l$, il existerait $g_l\in P_l$ tel que $g_l(t,x(t))\notin (t)^{{\b}^p_{h_l}(n,\l(n+1,d),i)}$. Notons alors $g:=(g_1...g_s)^e$. On a $g\in I$ par d\'efinition de $e$ et des $P_l$, mais $g(t,x(t))\notin (t)^{\b'}$ ce qui contredit ce qui pr\'ec\`ede, et donc  il existe un entier $l$ tel que $g(t,x(t))\in(t)^{{\b}_{h_l}(n,\l(n+1,d),i)}$ pour tout $g\in P_l$. Donc il existe $\ovl{x}(t)\in\k[[t]]^n$ tel que $g(t,\ovl{x}(t))=0$ pour tout $g\in P_l$ et tel que $\ovl{x}(t)-x(t)\in (t)^i$. Comme $I\subset P_l$, on obtient la conclusion voulue. On a montr\'e ainsi l'in\'egalit\'e (\ref{2}).\\
\\
$\bullet$
Nous supposons donc maintenant que $I=(f_1,...,f_p)$ est premier de hauteur $k$. Alors $\k[t,x]_I$ est r\'egulier, et on peut supposer que $f_1,...,f_k$ engendrent $I\k[t,x]_I$. Il existe donc un mineur d'ordre $k$ de la matrice jacobienne $\frac{\partial(f_1,...,f_k)}{\partial (t,x)}$, not\'e $\d$, tel que $\d\notin I$ (cf. par exemple proposition 1 \cite{Wa}). Notons $J:=(f_1,...,f_p,\d)$. On a alors ht$(J)=k+1$. Remarquons que $\deg(\d)\leq k(d-1)$. \\
Notons alors $I':=(f_1,...,f_k)$ et soit $I'=Q_1\cap\cdots\cap Q_q$ une d\'ecomposition primaire r\'eduite  de $I'$. Nous allons renum\'eroter les $Q_l$ de telle sorte que $Q_l\subset I$ pour $1\leq l\leq s$ et $Q_l\not\subset I$ pour $l>s$. Or $I\k[t,x]_I=I'\k[t,x]_I$ et $Q_1\k[t,x]_I\cap\cdots\cap Q_s\k[t,x]_I$ est une d\'ecomposition primaire r\'eduite de $I \k[t,x]_I$ (cf. Theorem 17, chapter 4 \cite{Z-S}), donc $s=1$ et $I=Q_1$. Soit $J':=Q_2\cap \cdots\cap Q_q$ si $q\neq 1$ et $J'=A$ si $q=1$. On pose alors $J=J'$ si $I\not\subset J'$ et $J:=A$ si $I\subset J'$. On  a donc $I'=I\cap J$ et $J\not\subset I$.\\
Chaque id\'eal $Q_l$ est engendr\'e par des polyn\^omes de degr\'e inf\'erieur ou \'egal \`a $\l_2(n+1,\,d)$. Donc $J'$ est engendr\'e par des polyn\^omes de degr\'e inf\'erieur ou \'egal \`a $(n+1)((q-2)\l_2(n+1,d))^{2^{n}}+\l_2(n+1,d)$.
On a $\d\in\sqrt{I+J}$ (cf. lemma 7 \cite{Wa}). Notons 
$$e':=(n+3)^2\left(1+\l_2(n+1,d)+(n+1)((d^{n+1}-2)\l_2(n+1,d))^{2^{n}}\right)^{n+2}.$$
Alors $\d^{e'}\in I+J$ car $q\leq d^{n+1}$.
Soit 
$$\a:=(e+1)'{\b}_{k+1}(n,k(d-1),i)+1$$
Soit $x(t)\in\k[[t]]^n$ tel que $f(t,x(t))\in (t)^{\a}$ pour tout $f\in I$.\\
\\
\textbf{Cas 1 :}  Si $\d(t,x(t))\in (t)^{{\b}_{k+1}(n,k(d-1),i)}$, comme $\a\geq {\b}_{k+1}(n,k(d-1),i)$, par d\'efinition de ${\b}_{k+1}$, il existe $\ovl{x}(t)\in\k[[t]]^n$ tel que $f(t,\ovl{x}(t))=\d(t,\ovl{x}(t))=0$ pour tout $f\in I$ et $\ovl{x}(t)-x(t)\in (t)^{i}$, et on a la conclusion voulue.\\
\\
\textbf{Cas 2 :} Supposons maintenant que $\d(t,x(t))\notin (t)^{{\b}_{k+1}(m,k(d-1),i)}$. En d\'erivant la relation  
$$f_r(t,x(t))\in(t)^{(e'+1){\b}_{k+1}(n,k(d-1),i)+1}$$
 par rapport \`a $t$, on obtient
$$\frac{\partial f_r}{\partial t}(t,x(t))=-\sum_{\l}\frac{\partial f_r}{\partial x_{\l}}(t,x(t))\frac{\partial x_{\l}(t)}{\partial t}\ \text{mod.}\ (t)^{(e'+1){\b}_{k+1}(n,k(d-1),i)}.$$ 
On en d\'eduit l'existence d'un mineur d'ordre $k$ de la matrice jacobienne de $f$, encore not\'e $\d$, qui ne fait intervenir que des d\'eriv\'ees partielles par rapport aux $x_{\l}$ (i.e. $\d=\frac{\partial(f_1,...,f_k)}{\partial (x_1,...,x_k)}$ quitte \`a renommer les variables $x_{\l}$), tel que $\d(t,x(t))\notin (t)^{{\b}_{k+1}(n,k(d-1),i)}$. En particulier $\d\notin (f_1,...,f_p)$.\\
Par d\'efinition de $\a$ et en utilisant le th\'eor\`eme des fonctions implicites de Tougeron (c.f. th\'eor\`eme 3.2 \cite{To} ou Lemma 5.11 \cite{Ar2}), il existe $\ovl{x}(t)\in\k[[t]]^n$ tel que $f_r(t,\ovl{x}(t))=0$, pour $1\leq r\leq k$, et $\ovl{x}(t)-x(t)\in (t)^{e'{\b}_{k+1}(n,k(d-1),i)}$. On a alors $\d(t,x(t))-\d(t,\ovl{x}(t))\in (t)^{e'{\b}_{k+1}(n,k(d-1),i)}$ et $e'\geq 1$. Donc 
$$\d(t,\ovl{x}(t))\notin (t)^{{\b}_{k+1}(n,k(d-1),i)}.$$ Or $\d^{e'}\in I+J$, donc $\d^{e'}=\sum_{r=1}^ph_rf_r+h_0$ o\`u $h_0\in J$. Comme $f_r(t,x(t))-f_r(t,\ovl{x}(t))\in (t)^{e'{\b}_{k+1}(n,k(d-1),i)}$ et $f_r(t,x(t))\in (t)^{\a}$, pour $1\leq r\leq p$, on voit que 
$$f_r(t,\ovl{x}(t))\in (t)^{e'{\b}_{k+1}(n,k(d-1),i)}\ \text{ pour } 1\leq r\leq p$$ (car $\a\geq e'{\b}_{k+1}(n,k(d-1),i)$). Donc n\'ecessairement $h_0(t,\ovl{x}(t))\notin  (t)^{e'{\b}_{k+1}(n,k(d-1),i)}$, d'o\`u $h_0(t,\ovl{x}(t))\neq 0$. Or $h_0f_r\in I\cap J=I'=(f_1,...,\,f_k)$ pour $1\leq r\leq p$. On en d\'eduit $f_r(t,\,\ovl{x}(t))=0$ pour $1\leq r\leq p$. Comme ${\b}_{k+1}(n,k(d-1),i)\geq i$, on a $\ovl{x}(t)-x(t)\in (t)^i$ ce qui conclut la d\'emonstration.

\section{Fonction de Artin d'un id\'eal binomial}\label{binomsec}
Pour \'etudier la fonction de Artin d'un id\'eal binomial, nous allons mettre les $x_i(t,z)$ sous forme de Weierstrass et nous allons utiliser le lemme suivant :
\begin{lemma}\label{w}
Soit $\k$ un corps quelconque.
Soient $P=u(x)(a_0(\wdt{x})+a_1(\wdt{x})x_n+\cdots+a_{d-1}(\wdt{x})x_n^{d-1}+x_n^d)$ et  $Q=v(x)(b_0(\wdt{x})+b_1(\wdt{x})x_n+\cdots+b_{e-1}(\wdt{x})x_n^{e-1}+x_n^e)\in\k[[x_1,...,x_n]]$ deux polyn\^omes de Weierstrass en $x_n$ (avec $x=(x_1,...,x_n)$ et $\wdt{x}:=(x_1,...,x_{n-1})$) tels que $P-Q\in (x)^i$ avec $i>d$. Alors $d=e$, $u(x)-v(x)\in (x)^{i-d}$ et $a_j(\wdt{x})-b_j(\wdt{x})\in (\wdt{x})^{i-d+\ord(P)-j}$ pour $0\leq j\leq d-1$.
\end{lemma}
\begin{proof}
Tout d'abord comme $i>d$, on voit que $e=d$ car $(P-Q)(0,...,0,x_n)\in (x_n)^i$ et le terme constant de $u$ est \'egal au terme constant de $v$. De m\^eme $\ord(P)=\ord(Q)$. D'autre part, si $P-Q\in (x)^i$ alors $u^{-1}P-u^{-1}Q\in (x)^i$. On peut donc supposer que $u=1$.\\
La division de Weirstrass de $x_n^d$ par $P$ par rapport \`a la variable $x_n$ est la suivante :
$$x_n^d=1.P+R(x)=P+(-a_{d-1}(\wdt{x})x_n^{d-1}-\cdots -a_0(\wdt{x})).$$
Consid\'erons la division de Weierstrass de $x_n^d$ par $Q$ par rapport \`a la variable $x_n$ : 
$$x_n^d=C(x)Q+R'(x)=C(x)Q+D_{d-1}(\wdt{x})x_n^{d-1}+\cdots +D_0(\wdt{x}).$$
Par unicit\'e dans le th\'eor\`eme de division de Weierstrass, on a $C(x)=v^{-1}(x)$ et $D_j(\wdt{x})=-b_j(\wdt{x})$ pour $0\leq j\leq d-1$. Cette division peut se faire de mani\`ere algorithmique. En effet on construit les suites $(C_k(x))_k$ et $(R_k(x))_k$ par r\'ecurrence de la mani\`ere qui suit : on pose $x_n^d=C_0(x)Q+R_0$ avec $C_0(x)=1$. Puis par induction, pour $k\geq 0$, si $x_n^d=C_k(x)Q+R_k(x)$, on consid\`ere le plus petit mon\^ome de $R_k$ divisible par $x_n^d$, not\'e $M_k$, et on pose $R_{k+1}(x):=R_k(x)-\frac{M_k}{x_n^d}Q$. Alors la suite $(\ord(R_{k+1}(x)-R_k(x)))_k$ est strictement croissante et $(R_k(x))_k$ converge vers $R'(x)$ pour la topologie $(x)$-adique. De m\^eme la suite $(\deg(M_k))_k$ est strictement croisssante et $C(x)=\sum_k\frac{M_k}{x_n^d}$. En particulier, comme $R_0(x)-R(x)=Q-P\in (x)^i$, $M_0$ est un mon\^ome de degr\'e sup\'erieur ou \'egal \`a $i$, et $C(x)-1\in (x)^{i-d}$. On en d\'eduit que $R(x)-R'(x)\in (x)^i$, donc $a_j(\wdt{x})-b_j(\wdt{x})\in (x)^{i-d+\ord(P)-j}$, pour $0\leq j\leq d-1$.
\end{proof}

Nous allons maintenant \'etudier la fonction de Artin d'un id\'eal binomial. Soit $\k$ un corps alg\'ebriquement clos et soit $I$ un id\'eal de $\k[[t,z]][X_1,...,X_n]$ engendr\'e par $f_1$,..., $f_p$ avec 
$$f_k:=a_kX^{\a_k}+b_kX^{\b_k}$$ o\`u $a_k$, $b_k\in\k$ et $\a_k$, $\b_k\in\N^n$ pour $1\leq k\leq p$. Consid\'erons $x_j(t,z)\in\k[[t,z]]$, $1\leq j\leq n$, tels que $f_k(x_j(t,z))\in (t,z)^i$, $1\leq k\leq p$. On peut supposer, quitte \`a faire un changement lin\'eaire de coordonn\'ees en $t$ et $z$, que les $x_j$ sont r\'eguli\`eres en la variable $z$. On a donc 
$$x_j(t,z)=u_j(t,z)(x_{j,0}(t)+x_{j,1}(t)z+\cdots+x_{j,d_{j}-1}(t)z^{d_j-1}+z^{d_j})$$
o\`u $d_j$ est l'ordre de $x_j(t,z)$ avec $x_{j,l}\in (t)\k[[t]]$, pour $1\leq j\leq n$ et $0\leq l\leq d_j$.\\
 Supposons que $i>\sum_j\a_{k,j}d_{j}$ pour tout $k$. En utilisant le lemme \ref{w}, on a alors 
$$D_k:=\sum_j\a_{k,j}d_{j}=\sum_j\b_{k,j}d_{j},\ \ 1\leq k\leq p,$$
$$f_k(u_1(t,z),...,u_n(t,z))\in (t,z)^{i-D_k},\ \ 1\leq k\leq p$$
et 
$$P_d(x_{j,l}(t))\in (t)^{i-d}$$
o\`u $P_d(X_{j,l})\in\k[X_{j,l},\  _{1\leq j\leq n,\ 0\leq l\leq d_j-1}]$ est le coefficient de $z^d$ dans 
$$\prod_j\left(X_{j,0}+\cdots+X_{j,d_{j}-1}z^{d_j-1}+z^{d_j}\right)^{\a_{k,j}}+\prod_j\left(X_{j,0}+\cdots+X_{j,d_{j}-1}z^{d_j-1}+z^{d_j}\right)^{\b_{k,j}}$$
Notons $D:=\max_kD_k$. On obtient alors deux syst\`emes d'\'equations ind\'ependants l'un de l'autre :
\begin{equation}\label{a}f_k(u_1(t,z),...,u_n(t,z))\in (t,z)^{i-D},\ \ 1\leq k\leq p\end{equation} 
\begin{equation}\label{b}P_d(x_{j,l}(t))\in (t)^{i-D},\ \ 0\leq d\leq D-1.\end{equation}
On va chercher la fonction d'approximation de Artin de ces deux syst\`emes, le premier ayant une fonction de Artin born\'ee par une fonction affine puisque  $u_i(0,0)\neq 0$ pour tout $i$ et que le lieu singulier d'une vari\'et\'e torique est inclus dans l'union des axes de coordonn\'ees,  et le second \'etant un syst\`eme \`a coefficients dans $\k[[t]]$.
\begin{theorem}\label{binom}
Soit $\k$ un corps alg\'ebriquement clos de caract\'eristique nulle. Alors on a les propri\'et\'es suivantes :
\begin{enumerate}
\item[i)] Pour tout $d'\in\N$ et pour tout $\underline{d}:=(d_1,....,d_n)\in\N^n$,  il existe  $a_{\underline{d},d'}$, $b_{\underline{d},d'}$ v\'erifiant la propri\'et\'e suivante :\\
Soit $I$ un id\'eal binomial  de $\k[U_1,...,U_n]$ engendr\'e par des bin\^omes $f_1,...,f_p$  de degr\'e inf\'erieur \`a $d'$. Soit $i\in\N$ et soient
$x_1(t,z)$,..., $x_n(t,z)\in\k[[t,z]]$ tels que $\ord(x_j(t,z))=d_j$ et $f_k(x_j(t,z))\in (t,z)^{a_{\underline{d},d'}i+b_{\underline{d},d'}}$, pour tous $k$. Alors il existe $\ovl{x}_j(t,z)\in\k[[t,z]]$, tels que $f_k(\ovl{x}_j(t,z))=0$ pour tout $k$ et tels que $\ovl{x}_j(t,z)-x_j(t,z)\in (t,z)^i$ pour tout $j$.\\
\item[ii)] Pour tout $d'\in\N$ il existe une fonction doublement exponentielle en $i$, not\'ee $\b_{d'}$, telle que  pour tout id\'eal binomial $I$ de $\k[U_1,...,U_n]$ engendr\'e par des bin\^omes de degr\'e inf\'erieur \`a $d'$, la fonction de Artin de $I\k[[t,z]][U]$ est born\'ee par $\b_{d'}$. \\
\end{enumerate}
\end{theorem}

\begin{proof}
Nous allons d'abord montrer i) dont on d\'eduira ensuite ii). \\
\\
Supposons, comme pr\'ec\'edemment, que $I$ est  engendr\'e par les $f_k:=a_kX^{\a_k}+b_kX^{\b_k}$, o\`u $a_k$, $b_k\in\k$ et $\a_k$, $\b_k\in\N^n$ pour $1\leq k\leq p$.\\
L'id\'eal $I$ \'etant un id\'eal binomial, $\sqrt{I}$ est encore un id\'eal binomial et, si $\sqrt{I}=I_1\cap\cdots\cap I_q$ est une d\'ecomposition primaire minimale de $\sqrt{I}$, alors les $I_k$ sont des id\'eaux binomiaux \cite{E-S}. Soit $e\in\N$ tel que $\sqrt{I}^e\subset I$.
Supposons que $f_l(u_1(t,z),...,u_n(t,z))\in (t,z)^{qei},\ \ 1\leq l\leq p$ o\`u $u_{j,0}:=u_j(0,0)\neq 0$ pour tout $j$. Alors $g(u_1(t,z),...,u_n(t,z))\in (t,z)^{qi}$ pour tout $g\in \sqrt{I}$, et donc il existe un entier $k$ tel que $g(u_1(t,z),...,u_n(t,z))\in (t,z)^{i}$ pour tout $g\in I_k$. Comme $I_k$ est un id\'eal premier binomial de $\k[U]$ alors $\left(\k[U]/I_k\right)_{(U_1-u_{1,0},...,U_n-u_{n,0})}$ est r\'egulier car $u_{1,0}...u_{n,0}\neq 0$ (le lieu singulier d'une vari\'et\'e torique est toujours inclus dans l'union des hyperplans de coordonn\'ees). Donc $\k\lgw \left(\k[U]/I_k\right)_{(U_1-u_{1,0},...,U_n-u_{n,0})}$ est lisse,  donc $\k[[t,z]]\lgw \left(\k[[t,z]][U]/I_k\right)_{(t,z,U_1-u_{1,0},...,U_n-u_{n,0})}$ est lisse ; ceci implique qu'il existe $\ovl{u}_j(t,z)\in\k[[t,z]]$, $1\leq j\leq n$, tels que 
$ g(\ovl{u}_1(t,z),...,\ovl{u}_n(t,z))=0$ pour tout $g\in I_k$ et $\ovl{u}_j(t,z)-u_j(t,z)\in (t,z)^i$.\\
\\
Soit $i\lgw ai+b$ la fonction de Artin de l'id\'eal de $\k[[t]][X_{j,l}]$ engendr\'e par les $P_d(X_{j,l})$, $0\leq d\leq D-1$. Posons $a_{\underline{d}}:=\max\{qe,a\}$ et $b_{\underline{d}}:=b+D$.  D'apr\`es les th\'eor\`emes \ref{bornes_alg} et  \ref{theo1}, on voit que $a_{\underline{d}}$ et $b_{\underline{d}}$ peuvent \^etre born\'ees par une fonction d\'ependant uniquement de $d'$.\\
Soit $i\in \N$, on a alors $a_{\underline{d}}i+b_{\underline{d}}>D$.\\
D'apr\`es ce qui pr\'ec\`ede, si $f_k(x_j(t,z))\in (t,z)^{a_{\underline{d}}i+b_{\underline{d}}} $ $\forall k$, alors 
$$f_k(u_1(t,z),...,u_n(t,z))\in (t,z)^{qei},\ \ 1\leq k\leq p$$
$$\text{et } P_d(x_{j,l}(t))\in (t)^{ai+b},\ \ 0\leq d\leq D.$$
Donc il existe $\ovl{u}_j(t,z)\in\k[[t,z]]$, $1\leq j\leq n$, tels que 
$ g(\ovl{u}_1(t,z),...,\ovl{u}_n(t,z))=0$ pour tout $g\in I$ et $\ovl{u}_j(t,z)-u_j(t,z)\in (t,z)^i$, et il existe $\ovl{x}_{j,l}(t)\in\k[[t]]$ tels que $P_d(x_{j,l}(t))=0$, $0\leq d\leq D$, et $\ovl{x}_{j,l}(t)-x_{j,l}(t)\in (t)^i$. On pose alors $\ovl{x}_j=\ovl{u}_j(t,z)(\ovl{x}{j,0}+\ovl{x}_{j,1}t+\cdots+\ovl{x}_{j,d_j-1}t^{d_j-1}+t^{d_j})$. On a bien $f_k(\ovl{x}_j(t,z))=0$, $1\leq k\leq p$, et $\ovl{x}_j(t,z)-x_j(t,z)\in (t,z)^i$, $1\leq j\leq n$. Ceci prouve i).\\
\\
On peut remarquer que $a$ et $b$ sont born\'ees par une fonction polynomiale en $\max_k\{|\a_k|,|\b_k|\}$ de degr\'e exponentiel en $\sum_jd_j$, d'apr\`es le th\'eor\`eme \ref{1var}. 
Donc il existe une constante $C$ telle que $a_{\underline{d}}, b_{\underline{d}}\leq C^{C^{\sum_jd_j}}$. Notons $a^{\mathcal{E}}_{\underline{d}}$ et $ b^{\mathcal{E}}_{\underline{d}}$ les plus petites constantes satisfaisant i) pour l'id\'eal engendr\'e par les $f_k$, $k\in\mathcal{E}$, o\`u  $\mathcal{E}$ est un sous-ensemble  de $\{1,...,p\}$. L\`a encore il existe une constante $C_{\mathcal{E}}$ telle que $C_{\mathcal{E}}^{C_{\mathcal{E}}^{\sum_{j\in\mathcal{E}}d_j}}$ majore  $a^{\mathcal{E}}_{\underline{d}}$ et $ b^{\mathcal{E}}_{\underline{d}}$. Donc en posant $C:=\max_{\mathcal{E}}C_{\mathcal{E}}$  on a  $a^{\mathcal{E}}_{\underline{d}}, b^{\mathcal{E}}_{\underline{d}}\leq C^{C^{\sum_{j}d_j}}$ pour tout sous-ensemble $\mathcal{E}$ de $\{1,...,p\}$.\\
Soit $x_j(t,z)\in\k[[t,z]]$, $1\leq j\leq n$, tels que $f_k(x_j(t,z))\in (t,z)^{C^{C^{ni}}(i+1)}$, $1\leq k\leq p$.
Si $d_j:=\ord(x_j(t,z))\geq i$, alors on pose $\ovl{x}_j(t,z)=0$, sinon on pose $\ovl{x}_j(t,z)=x_j(t,z)$. Alors $f_k(\ovl{x}_j(t,z))=0$ ou $f_k(\ovl{x}_j(t,z))=f_k(x_j(t,z))$ selon l'entier $k$. On peut donc remplacer les $x_j(t,z)$ par les $\ovl{x}_j(t,z)$ et supposer que $\ord(x_j(t,z))<i$ pour tout $j$. On supposer que $C^{C^{ni}}(i+1)>D:=\max_k\{\sum_j\a_{k,j}d_j, \sum_j\b_{k,j}d_j\}$ en choisissant $i$ assez grand. Comme $d_j<i$, $1\leq j\leq n$, on a $f_k(x_j(t,z))\in (t,z)^{a_{\underline{d}}i+b_{\underline{d}}} $, $1\leq k\leq p$. On applique alors i), et on a l'existence de $\ovl{x}_j(t,z)\in\k[[t,z]]$ tels que $f_k(\ovl{x}_j(t,z))=0$, $1\leq k\leq p$, et $\ovl{x}_j(t,z)-x_j(t,z)\in (t,z)^i$, $1\leq j\leq n$. Ceci prouve ii).
\end{proof}

\begin{ex}
Dans \cite{Ro2}, il est montr\'e que la fonction de Artin du polyn\^ome $X^2-ZY^2$ n'est  pas born\'ee par une fonction affine. On voit que la non-lin\'earit\'e de la fonction de Artin dans ce cas provient du fait que les fonctions de Artin-Greenberg des syst\`emes $P_d(X_{j,l})=0$, $0\leq d\leq D$, sont born\'ees par des fonctions affines dont les coefficients croissent vite en fonction de l'ordre des $x_{j,l}(t)$.
\end{ex}
\begin{ex}
On peut remarquer que la famille de solutions approch\'ees de l'\'equation $X^2-Y^3$ dans \cite{Ro1} sont des solutions dont l'ordre est born\'e.
\end{ex}

\begin{ex}
Soit $f:=X^2-Y^3$. Si $x(t,z),y(t,z)\in\k[[t,z]]$ v\'erifient $\ord(x(t,z))=3$ et $\ord(y(t,z))=2$, on peut \'ecrire 
$$x(t,z)=u(t,z)(x_0(t)+x_1(t)z+x_2(t)z^2+z^3)$$
$$y(t,z)=v(t,z)(y_0(t)+y_1(t)z+z^2).$$
Si $x^2(t,z)-y^3(t,z)\in (t,z)^i$ avec $i>6$, alors 
$$u^2(t,z)-v^3(t,z)\in (t,z)^{i-6}$$
et les $x_j(t)$ et $y_j(t)$ sont solutions du syst\`eme suivant modulo $(t)^{i-6}$ :
$$\left\{\begin{array}{c}
2x_2-3y_1 =0\\
x_2^2+2x_1-3y_1^2-3y_0 =0 \\
2x_0+2x_1x_2-y_1^3-6y_0y_1=0\ \\
x_1^2+2x_0x_2-3y_0y_1^2-3y_0^2=0\\
2x_0x_1-3y_0^2y_1 =0\\
x_0^2-y_0^3  =0\end{array}\right.$$
On peut v\'erifier \`a l'aide de Macaulay2 \cite{M2} que l'id\'eal de $\k[x_0,x_1,x_2,y_0,y_1]$ d\'efini par ces polyn\^omes n'est pas r\'eduit. Ceci semble assez g\'en\'eral. On peut remarquer que les vari\'et\'es alg\'ebriques d\'efinies par $P_d(X_{j,l})=0$, $0\leq d\leq D$, sont tr\`es proches des espaces de jets de la vari\'et\'e d\'efinie par $I$ qui ne sont pas r\'eduits en g\'en\'eral. 
\end{ex}

\begin{ex}
On consid\`ere ici le probl\`eme suivant : soient $p$ et $q$ deux entiers premiers entre eux. Si $x(t,z)$ est une s\'erie formelle telle que $
x^p(t,z)$ est proche d'une puissance $q$-i\`eme, est-ce que $x(t,z)$ est proche d'une puissance $q$-i\`eme ? Y a-t-il une fonction qui mesure le rapport entre la distance de $x^p(t,z)$ \`a une puissance $q$-i\`eme et celle de $x(t,z)$ \`a une puissance $q$-i\`eme ? On a la r\'eponse suivante :
\begin{prop}
Soient $p$ et $q$ premiers entre eux. Il existe une fonction $\b : \N\lgw \N$ telle que pour tout $x\in\k[[t,z]]$, si il existe $u(t,z)\in\k[[t,z]]$ telle que $x^p(t,z)-u^q(t,z)\in (t,z)^{\b(i)}$, alors il existe $v(t,z)\in\k[[t,z]]$ telle que $x(t,z)-v^q(t,z)\in (t,z)^i$.\\ On peut choisir pour $\b$ une fonction doublement exponentielle en $i$.\\ Pour tout $d\in\N$, si on se restreint \`a tous les $x(t,z)$ dont l'ordre vaut $d$, alors on peut choisir pour $\b$ une fonction affine.
\end{prop}

\begin{proof}
Soit $\b$ la fonction de Artin de $X^p-Y^q$.
Supposons que $x^p(t,z)-u^q(t,z)\in (t,z)^{\b(i)}$. Alors il existe $x'(t,z)$, $u'(t,z)\in\k[[t,z]]$ tels que $x'^p(t,z)-u'^q(t,z)=0$ et $x(t,z)-x't(t,z)$, $u(t,z)-u'(t,z)\in (t,z)^i$. Comme $p$ et $q$ sont premiers entre eux et que $\k[[t,z]]$ est factoriel il existe $v(t,z)\in\k[[t,z]]$ tel que 
$x'(t,z)=v^q(t,z)$ et $u'(t,z)=v^p(t,z)$. On a alors le r\'esultat avec le th\'eor\`eme \ref{binom}.

\end{proof}
\end{ex}

\end{document}